\documentclass[12pt,reqno]{amsart}
\usepackage[top=2cm,bottom=2cm,right=2.5cm,left=2.5cm]{geometry}
\usepackage{amssymb}
\usepackage{amsmath, amsthm, amscd, amsfonts, amssymb, graphicx, color}
\usepackage[bookmarksnumbered, colorlinks, plainpages]{hyperref}
\hypersetup{colorlinks=true,linkcolor=red, anchorcolor=green, citecolor=cyan, urlcolor=red, filecolor=magenta, pdftoolbar=true}
 
\usepackage{hyperref}

\newtheorem{theorem}{Theorem}[section]
\newtheorem{lemma}[theorem]{Lemma}
\newtheorem{proposition}[theorem]{Proposition}

\theoremstyle{definition}
\newtheorem{definition}[theorem]{Definition}
\newtheorem{example}[theorem]{Example}

\theoremstyle{remark}
\newtheorem{remark}[theorem]{Remark}

\numberwithin{equation}{section}

\newcommand{\A}{\mathcal{A}}       
\newcommand{\M}{\mathcal{H}}       
 
\newcommand{\ip}[2]{\langle #1,\,#2\rangle}
     
\DeclareMathOperator{\id}{id}
\newcommand{\otimesA}{\otimes_\A} 

\begin{document}
	
	\setcounter{page}{1}
	
	\title[Continuous biframes in Hilbert $C^{\ast}-$modules]{Continuous biframes in Hilbert $C^{\ast}-$modules}
	
   \author[A. Lfounoune, A. Karara, M. Rossafi]{Abdellatif Lfounoune$^{1,*}$ Abdelilah Karara$^{1}$ and Mohamed Rossafi$^{2}$}
   
   	\dedicatory{This paper is dedicated to Professor Samir Kabbaj}
	
	\address{$^{1}$Laboratory Analysis, Geometry and Applications, University of Ibn Tofail, Kenitra, Morocco}
	\email{\textcolor[rgb]{0.00,0.00,0.84}{abdellatif.lfounoune@uit.ac.ma; abdelilah.karara@uit.ac.ma}}
	
	\address{$^{2}$Laboratory Analysis, Geometry and Applications, Higher School of Education and Training, University of Ibn Tofail, Kenitra, Morocco}
	\email{\textcolor[rgb]{0.00,0.00,0.84}{mohamed.rossafi@uit.ac.ma}}
	
	
	
	\begin{abstract}
In this paper, we will introduce the concept of a continuous biframe for Hilbert $ C^{\ast}- $modules. Then, we examine some characterizations of this biframe with the help of an invertible and adjointable operator is given. Moreover, we study continuous biframe Bessel multiplier and dual continuous biframe in Hilbert $ C^{\ast}- $modules. Also, we develop the concept of continuous biframes in the tensor product of two Hilbert $C^{\ast}$-modules over a unital $C^{\ast}$-algebra $\mathcal{A}$ and provide some properties of invertible transformed biframes and Bessel multipliers in the tensor product.
\end{abstract}
\maketitle

\section{Introduction and preliminaries}
 
\smallskip\hspace{.6 cm} 
Frames for Hilbert spaces were introduced by Duffin and Schaefer \cite{Duf} in 1952 to study some deep problems in nonharmonic Fourier series by abstracting the fundamental notion of Gabor \cite{Gab} for signal processing. In fact, in $1946$ Gabor, showed that any function $f\in L^{2}(\mathbb{R})$ can be reconstructed via a Gabor system $\{g(x-ka)e^{2\pi imbx}: k,m\in\mathbb{Z}\}$ where $g$ is a continuous compact support function. These ideas did not generate much interest outside of nonharmonic Fourier series and signal processing until the landmark paper of Daubechies, Grossmannn, and Meyer \cite{DGM} in 1986, where they developed the class of tight frames for signal reconstruction and they showed that frames can be used to find series expansions of functions in $L^{2}(\mathbb{R})$ which are very similar to the expansions using orthonormal bases. After this innovative work the theory of frames began to be widely studied. While orthonormal bases have been widely used for many applications, it is the redundancy that makes frames useful in applications.

The idea of pair frames, which refers to a pair of sequences in a Hilbert space, was first presented in \cite{Fer}. Parizi, Alijani and Dehghan \cite{MF} studied Biframe, which is a generalization of controlled frame in Hilbert space. The concept of a frame is defined from a single sequence but to define a biframe we will need two sequences. In fact, the concept of biframe is a generalization of controlled frames and a special case of pair frames. For more detailed information on biframes theory, readers are recommended to consult: \cite{Assila, El Jazzar, Ghiati, Karara, Lfounoune, Lfounoune1, Massit, Echarghaoui, Rossafi1, Rossafi2, RFDCA, Kabbaj,  Chouchene, NhariRossTou, r1, r3, r5, r6}.
  
In this paper, we will introduce the concept of continuous biframes in Hilbert $C^{\ast}-$modules and we present some examples of this type of frame. Moreover, we investigate a characterization of continuous biframe using its frame operator is established. Also,  we study continuous biframe Bessel multiplier and dual continuous biframe in Hilbert $ C^{\ast}- $modules. Finally, we introduce the concept of continuous biframes in the tensor product of two Hilbert $C^{\ast}$-modules over a unital $C^{\ast}$-algebra $\mathcal{A}$ and investigate the properties of invertible transformed biframes and Bessel multipliers within this framework. 

Hilbert $C^{\ast}$-modules are generalizations of Hilbert spaces by allowing the inner product to take values in a $C^{\ast}$-algebra rather than in the field of complex
numbers.

Let's now review the definition of a Hilbert $C^{\ast}$-module with basic properties and some facts concerning operators on Hilbert $C^{\ast}$-module.

\begin{definition} \cite{Kal}
		Let $\mathcal{A}$ be a unital $ C^{\ast}- $ algebra and $ \mathcal{H} $ be a left $ \mathcal{A}-$ module, such that the linear structures of $ \mathcal{A} $ and $ \mathcal{H} $ are compatible. $ \mathcal{H} $ is a pre-Hilbert $ \mathcal{A} $ module if $ \mathcal{H} $ is equipped with an $ \mathcal{A}- $valued inner product $ \langle \cdot, \cdot\rangle_{\mathcal{A}} : \mathcal{H} \times \mathcal{H}\rightarrow \mathcal{A}$ such that is sesquilinear, positive definite and respects the module action. In the other words,
		\begin{itemize}
			\item [1]- $\langle x, x\rangle_{\mathcal{A}} \geq 0$,    $ \forall x\in \mathcal{H}$ and $\langle x, x\rangle_{\mathcal{A}} = 0$ if and only if $ x=0 $.
			\item [2]- $\langle ax+y, z\rangle_{\mathcal{A}} =a\langle x,z\rangle_{\mathcal{A}}+\langle y,z\rangle_{\mathcal{A}} $ for all $ a\in \mathcal{A} $ and $ x,y,z \in \mathcal{H}. $
		\item[3]- $ \langle x,y \rangle_{\mathcal{A}} =  \langle y,x \rangle_{\mathcal{A}}^{\ast}$ for all $ x,y \in\mathcal{H} $.
		\end{itemize}
	For $ x\in\mathcal{H} $, we define $\Vert x\Vert = \Vert \langle x,x \rangle_{\mathcal{A}} \Vert^{\frac{1}{2}}$. If $ \mathcal{H} $ is complete with $ \Vert .\Vert $, it is called a Hilbert $ \mathcal{A}- $module or a Hilbert $ C^{\ast} -$module over $ \mathcal{A} $. For every $a$ in $ C^{\ast}- $algebra $ \mathcal{A} $, we have  $ \vert a \vert = (a^{\ast}a)^{\frac{1}{2}} $ and the $ \mathcal{A}- $valued norm on $ \mathcal{H} $ is defined by $ \vert x\vert= \langle x,x\rangle_{\mathcal{A}}^{\frac{1}{2}} $ for $ x\in\mathcal{H} $.
	\end{definition}
	\begin{lemma} \cite{Pas}
	Let $\mathcal{H}$ be a Hilbert $\mathcal{A}$-module. If $\mathcal{T}\in End_{\mathcal{A}}^{\ast}(\mathcal{H})$, then $$\langle \mathcal{T}x,\mathcal{T}x\rangle_{\mathcal{A}}\leq\|\mathcal{T}\|^{2}\langle x,x\rangle_{\mathcal{A}},\quad \forall x\in\mathcal{H}.$$
\end{lemma}
Let $\mathcal{F}$ denote a Banach space, and $(\Omega, \mu)$ represent a measure space  with
 positive measure $\mu$, with $f: \Omega \rightarrow \mathcal{F}$ being a measurable function.  Integral of the Banach-valued function $f$ has been defined by Bochner and others. Most properties of this integral are similar to those of the integral of real-valued functions. Because every $C^{\ast}$-algebra and Hilbert $C^{\ast}$-module is a Banach space thus we can use this integral and its properties.

 Throughout, we assume $\mathcal{A}$ is a unital $C^{\ast}$-algebra, $\mathcal{H}$ is a Hilbert $C^{\ast}$-module over $\mathcal{A}$, and $(\Omega, \mu)$ is a measure space. Define,
$$
L^2(\Omega, \mathcal{A})=\left\{\varphi: \Omega \rightarrow \mathcal{A} \quad ; \quad\left\|\int_{\Omega}\left|(\varphi(\omega))^*\right|^2 d \mu(\omega)\right\|<\infty\right\} .
$$
For every $\varphi, \psi \in L^2(\Omega, \mathcal{A})$, if the inner product is defined by 
$$
\langle\varphi, \psi\rangle_{\mathcal{A}}=\int_{\Omega}\langle\varphi(\omega), \psi(\omega)\rangle_{\mathcal{A}} d \mu(\omega).
$$
The norm is defined by $\|\varphi\|=\|\langle\varphi, \varphi\rangle_{\mathcal{A}}\|^{\frac{1}{2}}$, then $L^2(\Omega, \mathcal{A})$ is a Hilbert $\mathcal{A}$-module \cite{Lance}.
\begin{definition}
Let $\mathcal{H}$ be Hilbert $\mathcal{A}$-module and $(\Omega, \mu)$ be a measure space with positive measure $\mu$. A mapping $\mathcal{X} : \Omega \to \mathcal{H}$ is called a continuous frame with respect to $\left(\Omega, \mu\right)$ if
\begin{itemize}
\item[$(i)$] $\mathcal{X}$ is weakly-measurable, i.e., for all $f \in \mathcal{H}$, $\omega \mapsto \langle  f, \mathcal{X}(\omega)\rangle_{\mathcal{A}} $ is a measurable function on $\Omega$,
\item[$(ii)$]There exist constants $0 < A \leq B < \infty$ such that
\end{itemize}
\[A \langle f, f\rangle_{\mathcal{A}} \leq \int_{\Omega}\langle f, \mathcal{X}(\omega)\rangle_{\mathcal{A}} \langle \mathcal{X}(\omega), f\rangle_{\mathcal{A}} d\mu \leq B\langle f, f\rangle_{\mathcal{A}},\]
for all $f \in \mathcal{H}$. 
\end{definition}
The constants $A$ and $B$ are called continuous frame bounds. If $A = B$, then it is called a tight continuous frame. If the mapping $\mathcal{X}$ satisfies only the right inequality, then it is called continuous Bessel mapping with Bessel bound $B$.

Let $\mathcal{X} : \Omega \to \mathcal{H}$ be a continuous frame for $\mathcal{H}$. Then the synthesis operator $\mathcal{T}_{\mathcal{X}} : L^{2}\left(\Omega,\mu\right) \to \mathcal{H}$ weakly defined by
\[\langle  \mathcal{T}_{\mathcal{X}}(\varphi), h\rangle_{\mathcal{A}}  = \int_{\Omega}\varphi(\omega)\langle  \mathcal{X}(\omega), f\rangle_{\mathcal{A}} d\mu,\]where $\varphi \in L^{2}\left(\Omega,\mu\right)$ and $f \in \mathcal{H}$ and its adjoint operator called the analysis operator $\mathcal{T}^{\ast}_{\mathcal{X}} : \mathcal{H} \to L^{2}\left(\Omega,\mu\right)$ is given by  
\[\mathcal{T}^{\ast}_{\mathcal{X}}\mathcal{X}(\omega) = \langle  f, \mathcal{X}(\omega)\rangle_{\mathcal{A}} \;,\; f \in \mathcal{H},\;\; \omega \in \Omega.\]

The frame operator $S_{\mathcal{X}} : \mathcal{H} \to \mathcal{H}$ is weakly defined by
\[\langle  S_{\mathcal{X}}f, f\rangle_{\mathcal{A}}  = \int_{\Omega}\langle  f, \mathcal{X}(\omega)\rangle_{\mathcal{A}} \langle  \mathcal{X}(\omega), f\rangle_{\mathcal{A}} d\mu, \;\forall f \in \mathcal{H}.\]

Let $GL^{+}(\mathcal{H})$ be the set of all positive bounded linear invertible operators on $\mathcal{H}$ with bounded
inverse. We reserve the notation $\operatorname{End}_\mathcal{A}^*(\mathcal{H})$ for the set of all adjointable operators from $\mathcal{H}$ to $\mathcal{H}$.

\section{Continuous biframe in Hilbert $C^{\ast}$-modules}

In this section, we begin by presenting the definition of a continuous biframe in a Hilbert $C^{\ast}-$modules, followed by a discussion of some of its properties. 

\begin{definition}\label{def1.01}
A pair $(\mathcal{X}, \mathcal{Y}) = \left(\mathcal{X} : \Omega \to \mathcal{H},\; \mathcal{Y} : \Omega \to \mathcal{H}\right)$ of mappings is called a continuous biframe for $\mathcal{H}$ with respect to $\left(\Omega, \mu\right)$ if:
\begin{itemize}
\item[$(i)$] $\mathcal{X}, \mathcal{Y}$ are weakly-measurable, i.e., for all $f \in \mathcal{H}$, $\omega \mapsto \langle   f, \mathcal{X}(\omega)\rangle_{\mathcal{A}} $ and $\omega \mapsto \langle  f, \mathcal{Y}(\omega)\rangle_{\mathcal{A}} $ are measurable functions on $\Omega$,
\item[$(ii)$]there exist constants $0 < A \leq B < \infty$ such that for all $f \in \mathcal{H}$,
\end{itemize}
\begin{align}
A\langle f, f\rangle_{\mathcal{A}} \leq \int_{\Omega}\langle  f, \mathcal{X}(\omega)\rangle_{\mathcal{A}} \langle  \mathcal{Y}(\omega), f\rangle_{\mathcal{A}} d\mu \leq B\langle f, f\rangle_{\mathcal{A}}.\label{3.eqq3.11}
\end{align}
The constants $A$ and $B$ are called continuous biframe bounds. If $A = B$, then it is called a tight continuous biframe and if $A = B = 1$, then it is called Parseval continuous biframe . 

If $(\mathcal{X}, \mathcal{Y})$ satisfies only the right inequality (\ref{3.eqq3.11}), then it is called continuous biframe Bessel mapping with Bessel bound $B$. 
\end{definition}
\begin{remark}
Let $\mathcal{X} : \Omega \to \mathcal{H}$ be a mapping. Consequently, in light of the Definition \ref{def1.01}, we express that
\begin{itemize}
\item[$(i)$]If $(\mathcal{X}, \mathcal{X})$ is a continuous biframe for $\mathcal{H}$, then $\mathcal{X}$ is a continuous frame for $\mathcal{H}$.
\item[$(ii)$]If $P \in GL^{+}(\mathcal{H})$, $(\mathcal{X}, P\mathcal{X})$ is a continuous biframe for $\mathcal{H}$, then $\mathcal{X}$ is a $P$--controlled continuous frame for $\mathcal{H}$,
\item[$(iii)$]If $P, Q \in GL^{+}(\mathcal{H})$, $(P\mathcal{X}, Q\mathcal{X})$ is a continuous biframe for $\mathcal{H}$, then $\mathcal{X}$ is a $(P, Q)$--controlled continuous frame for $\mathcal{H}$.  
\end{itemize} 
\end{remark}
We now provide an example that verifies the description given above.
\begin{example}  Assume that $\mathcal{A}=\left\{\left(\begin{array}{ll}a & 0 \\ 0 & b\end{array}\right): a, b \in \mathbb{C}\right\}$, then $\mathcal{A}$ is a unital $C^*$-algebra. Also $\mathcal{A}$ is a Hilbert $C^*$-module over itself, with the inner product:
$$
\begin{aligned}
\langle., .\rangle_{\mathcal{A}}: \mathcal{A} \times \mathcal{A} & \rightarrow \mathcal{A} \\
(M, N) & \longmapsto M(\bar{N})^t .
\end{aligned}
$$
Assume that $(\Omega, \mu)$ is a measure space where $\Omega=[0,1]$ and $\mu$ is the Lebesgue measure. Define $\mathcal{X} : \Omega \to \mathcal{A}$ by
\[
\mathcal{X}(\omega) = 
\begin{pmatrix}
\;2 \omega & 0\\
\;0           &   1-\omega\\
\end{pmatrix}
,\; \omega \in \Omega
,\] 
and $\mathcal{Y} : \Omega \to \mathcal{A}$ by
\[
\mathcal{Y}(\omega) = 
\begin{pmatrix}
\;3\omega & 0\\
\;0           & \omega+1\\
\end{pmatrix}
,\; \omega \in \Omega.
\]
For every $f=\left(\begin{array}{ll}a & 0 \\ 0 & b\end{array}\right) \in \mathcal{A}$, we have
\begin{small}
$$
\begin{aligned}
\int_{\Omega}\langle f, \mathcal{X}(\omega)\rangle_{\mathcal{A}}\langle \mathcal{Y}(\omega), f\rangle_{\mathcal{A}} d \mu(\omega) & =\int_{[0,1]}\left\langle\left(\begin{array}{cc}
a & 0 \\
0 & b
\end{array}\right),\left(\begin{array}{cc}
2 \omega & 0 \\
0 & 1-\omega
\end{array}\right)\right\rangle_{\mathcal{A}}\left\langle\left(\begin{array}{cc}
3 \omega & 0 \\
0 & \omega+1
\end{array}\right),\left(\begin{array}{ll}
a & 0 \\
0 & b
\end{array}\right)\right\rangle_{\mathcal{A}} d \mu(\omega) \\
& =\int_{[0,1]}\left(\begin{array}{cc}
2 \omega a & 0 \\
0 & (1-\omega) b
\end{array}\right)\left(\begin{array}{cc}
3 \omega \bar{a} & 0 \\
0 & (\omega+1) \bar{b}
\end{array}\right) d \mu(\omega) \\
& =\int_{[0,1]}\left(\begin{array}{cc}
6 \omega^2 & 0 \\
0 & 1-\omega^2
\end{array}\right)\left(\begin{array}{cc}
|a|^2 & 0 \\
0 & |b|^2
\end{array}\right) d \mu(\omega) \\
& =\left(\begin{array}{cc}
|a|^2 & 0 \\
0 & |b|^2
\end{array}\right) \int_{[0,1]}\left(\begin{array}{cc}
6 \omega^2 & 0 \\
0 & 1-\omega^2
\end{array}\right) d \mu(\omega) \\
& =\left(\begin{array}{cc}
2 & 0 \\
0 & \frac{2}{3}
\end{array}\right)\left(\begin{array}{cc}
|a|^2 & 0 \\
0 & |b|^2
\end{array}\right)
\end{aligned}
$$
\end{small}

Consequently,
 $$ \dfrac{2}{3}\langle f, f\rangle_{\mathcal{A}} \leq \int_{\Omega}\langle  f, \mathcal{X}(\omega)\rangle_{\mathcal{A}} \langle  \mathcal{Y}(\omega),f \rangle_{\mathcal{A}} d\mu \leq 2 \langle f, f\rangle_{\mathcal{A}} .$$
Therefore, $(\mathcal{X}, \mathcal{Y})$ is a continuous biframe for $\mathcal{H}$ with bound $\dfrac{2}{3}$ and $2$. 
\end{example}

Next, we introduce the continuous biframe operator and provide some of its properties.
\begin{definition}
Let $(\mathcal{X}, \mathcal{Y}) = \left(\mathcal{X} : \Omega \to \mathcal{H},\; \mathcal{Y} : \Omega \to \mathcal{H}\right)$ be a continuous biframe for $\mathcal{H}$ with respect to $\left(\Omega, \mu\right)$. Then the continuous biframe operator $S_{\mathcal{X}, \mathcal{Y}} : \mathcal{H} \to \mathcal{H}$ is defined by
\[S_{\mathcal{X}, \mathcal{Y}}f = \int_{\Omega}\langle  f, \mathcal{X}(\omega)\rangle_{\mathcal{A}} \mathcal{Y}(\omega)d\mu,\]
for all $f \in \mathcal{H}$.
\end{definition}

For every $f \in \mathcal{H}$, we have 
\begin{align}
\langle  S_{\mathcal{X}, \mathcal{Y}}f, f\rangle_{\mathcal{A}}  = \int_{\Omega}\langle  f, \mathcal{X}(\omega)\rangle_{\mathcal{A}} \langle  \mathcal{Y}(\omega), f\rangle_{\mathcal{A}} d\mu.\label{3.eqq3.12}
\end{align}
This implies that, for each $f \in \mathcal{H}$,
\[A\langle f, f\rangle_{\mathcal{A}} \leq \langle  S_{\mathcal{X}, \mathcal{Y}}f, f\rangle_{\mathcal{A}}  \leq  B\langle f, f\rangle_{\mathcal{A}}.\]
Hence, $AI \leq S_{\mathcal{X}, \mathcal{Y}} \leq BI$, where $I$ is the identity operator on $\mathcal{H}$. consequently, $S_{\mathcal{X}, \mathcal{Y}}$ is positive and invertible.

\begin{theorem}\label{t36}
Let $S_{\mathcal{X}, \mathcal{Y}}$ the continuous biframe operator, if the pair $(\mathcal{X}, \mathcal{Y})$ is a continuous biframe for $\mathcal{H}$ with respect to $\left(\Omega, \mu\right)$ Then $S_{\mathcal{X}, \mathcal{Y}}$ is adjointable and $S^{\ast}_{\mathcal{X}, \mathcal{Y}}=S_{\mathcal{Y}, \mathcal{X}}$.
\end{theorem}

\begin{proof}

From (\ref{3.eqq3.12}), we can write
\begin{align*}
 \langle  S_{\mathcal{X}, \mathcal{Y}}f, f\rangle_{\mathcal{A}} &= \int_{\Omega}\langle  f, \mathcal{X}(\omega)\rangle_{\mathcal{A}} \langle  \mathcal{Y}(\omega), f\rangle_{\mathcal{A}} d\mu  \\    
&=\langle \int_{\Omega}\langle  f, \mathcal{X}(\omega)\rangle_{\mathcal{A}} \mathcal{Y}(\omega)d\mu , f\rangle_{\mathcal{A}}\\
&= \int_{\Omega}\langle  f, \mathcal{X}(\omega)\rangle_{\mathcal{A}} \langle  \mathcal{Y}(\omega), f\rangle_{\mathcal{A}} d\mu\\
&= \int_{\Omega}\left(\langle  f, \mathcal{Y}(\omega)\rangle_{\mathcal{A}} \langle  \mathcal{X}(\omega), f\rangle_{\mathcal{A}}\right)^{\ast} d\mu\\
&= \left(\int_{\Omega}\langle  f, \mathcal{Y}(\omega)\rangle_{\mathcal{A}} \langle  \mathcal{X}(\omega), f\rangle_{\mathcal{A}} d\mu\right)^{\ast}
\\
&= \langle  S_{\mathcal{Y}, \mathcal{X}}f, f\rangle_{\mathcal{A}}^{\ast} \\
&=\langle  f,S_{\mathcal{Y}, \mathcal{X}} f\rangle_{\mathcal{A}}
.
\end{align*}
Therefore, $S_{\mathcal{X}, \mathcal{Y}}$ is adjointable and $S^{\ast}_{\mathcal{X}, \mathcal{Y}}=S_{\mathcal{Y}, \mathcal{X}}$.
\end{proof}
\begin{proposition}
Let $S_{\mathcal{X}, \mathcal{Y}}$ and $S_{\mathcal{Y}, \mathcal{X}}$ be continuous biframe operators such that $S_{\mathcal{X}, \mathcal{Y}}=S_{\mathcal{Y}, \mathcal{X}}$. Then the pair $(\mathcal{X}, \mathcal{Y})$ is a continuous biframe for $\mathcal{H}$ with respect to $\left(\Omega, \mu\right)$ if and only if  $(\mathcal{Y}, \mathcal{X})$ is a continuous biframe for $\mathcal{H}$ with respect to $\left(\Omega, \mu\right)$.
\end{proposition}
\begin{proof}
Let $(\mathcal{X}, \mathcal{Y})$ is a continuous biframe for $\mathcal{H}$ with bounds $A$ and $B$. Then for every $f \in \mathcal{H}$, we have 
\[A\langle f, f\rangle_{\mathcal{A}} \leq\langle S_{\mathcal{X}, \mathcal{Y}}f,f\rangle_{\mathcal{A}}= \int_{\Omega}\langle  f, \mathcal{X}(\omega)\rangle_{\mathcal{A}} \langle  \mathcal{Y}(\omega), f\rangle_{\mathcal{A}} d\mu \leq B\langle f, f\rangle_{\mathcal{A}}.\] 
Since $S_{\mathcal{X}, \mathcal{Y}}=S_{\mathcal{Y}, \mathcal{X}}$ we have 
$$\langle S_{\mathcal{Y}, \mathcal{X}}f,f\rangle_{\mathcal{A}}=\langle S_{\mathcal{X}, \mathcal{Y}}f,f\rangle_{\mathcal{A}}=\int_{\Omega}\langle  f, \mathcal{Y}(\omega)\rangle_{\mathcal{A}} \langle  \mathcal{X}(\omega), f\rangle_{\mathcal{A}} d\mu$$
Thus, for each $f \in \mathcal{H}$, we have 
\[A\langle f, f\rangle_{\mathcal{A}} \leq \int_{\Omega}\langle  f, \mathcal{Y}(\omega)\rangle_{\mathcal{A}} \langle  \mathcal{X}(\omega), f\rangle_{\mathcal{A}} d\mu \leq B\langle f, f\rangle_{\mathcal{A}}.\]
Therefore, $(\mathcal{Y}, \mathcal{X})$ is a continuous biframe for $\mathcal{H}$.

Likewise, we can establish the converse part of this theorem.
\end{proof}
We suppose that  $S_{\mathcal{X}, \mathcal{Y}}$ is self-adjoint operator. Thus, every $f \in \mathcal{H}$ has the representations
\begin{align*}
&f = S_{\mathcal{X}, \mathcal{Y}} S^{- 1}_{\mathcal{X}, \mathcal{Y}}f = \int_{\Omega}\langle  f, S^{- 1}_{\mathcal{X}, \mathcal{Y}}\mathcal{X}(\omega)\rangle_{\mathcal{A}} \mathcal{Y}(\omega)d\mu,\\
&f = S^{- 1}_{\mathcal{X}, \mathcal{Y}}S_{\mathcal{X}, \mathcal{Y}}f = \int_{\Omega}\langle  f, \mathcal{X}(\omega)\rangle_{\mathcal{A}} S^{- 1}_{\mathcal{X}, \mathcal{Y}}\mathcal{Y}(\omega)d\mu.
\end{align*}

In the following theorem, we establish a characterization of a continuous biframe by utilizing its biframe operator.

\begin{theorem}
Let $(\mathcal{X}, \mathcal{Y})$ is a continuous biframe Bessel mapping for $\mathcal{H}$ with respect to $\left(\Omega, \mu\right)$.Then $(\mathcal{X}, \mathcal{Y})$ is a continuous biframe  with bounds $A$ and $B$ for $\mathcal{H}$ if and only if $S_{\mathcal{X}, \mathcal{Y}} \geq A I$, where $S_{\mathcal{X}, \mathcal{Y}}$ is the continuous biframe operator for $(\mathcal{X}, \mathcal{Y})$.
\end{theorem}

\begin{proof}
Let $(\mathcal{X}, \mathcal{Y})$ is a continuous biframe for $\mathcal{H}$ with bounds $A$ and $B$. Then using (\ref{3.eqq3.11}) and (\ref{3.eqq3.12}), for each $f \in \mathcal{H}$, we get 
\[A\langle f, f\rangle_{\mathcal{A}} \leq \langle  S_{\mathcal{X}, \mathcal{Y}}f, f\rangle_{\mathcal{A}}  = \int_{\Omega}\langle  f, \mathcal{X}(\omega)\rangle_{\mathcal{A}} \langle  \mathcal{Y}(\omega), f\rangle_{\mathcal{A}} d\mu \leq  B\langle f, f\rangle_{\mathcal{A}}.\]
Thus,
\[A\langle  f, f\rangle_{\mathcal{A}}  \leq \langle  S_{\mathcal{X}, \mathcal{Y}}f, f\rangle_{\mathcal{A}}  ,\]
Hence, $$S_{\mathcal{X}, \mathcal{Y}} \geq A I.$$

Conversely, assume that $S_{\mathcal{X}, \mathcal{Y}} \geq A I$. Then, for every $f \in \mathcal{H}$, we have
\[A\langle  f, f\rangle_{\mathcal{A}} \leq  \langle  S_{\mathcal{X}, \mathcal{Y}}f, f\rangle_{\mathcal{A}}  = \int_{\Omega}\langle  f, \mathcal{X}(\omega)\rangle_{\mathcal{A}} \langle  \mathcal{Y}(\omega), f\rangle_{\mathcal{A}} d\mu.\] 
Since $(\mathcal{X}, \mathcal{Y})$ is a continuous biframe Bessel mapping for $\mathcal{H}$. Therefore, $(\mathcal{X}, \mathcal{Y})$ is a continuous biframe for $\mathcal{H}$.
\end{proof}

Furthermore, we provide a characterization of a continuous biframe with the assistance of an invertible operator on $\mathcal{H}$.
\begin{theorem}\label{3.thm3.39}
Let $\mathcal{T}\in \operatorname{End}_\mathcal{A}^*(\mathcal{H})$ be invertible on $\mathcal{H}$. Then the
following statements are equivalent:
\begin{enumerate}
\item $(\mathcal{X}, \mathcal{Y})$ is a continuous biframe for $\mathcal{H}$ with respect to $\left(\Omega, \mu\right)$
\item $(\mathcal{T}\mathcal{X}, \mathcal{T}\mathcal{Y})$ is a continuous biframe for $\mathcal{H}$ with respect to $\left(\Omega, \mu\right)$.
\end{enumerate}
\end{theorem}

\begin{proof}
(1)$\Rightarrow$(2) For each $f \in \mathcal{H}$, $$ \omega \mapsto\langle  f, \mathcal{T}\mathcal{X}(\omega)\rangle_{\mathcal{A}}   $$ and $$\omega \mapsto \langle  f, \mathcal{T}\mathcal{Y}(\omega)\rangle_{\mathcal{A}}  $$ are measurable functions on $\Omega$. Let $(\mathcal{X}, \mathcal{Y})$ is a continuous biframe for $\mathcal{H}$ with bounds $A$ and $B$ and $\mathcal{T}\in \operatorname{End}_\mathcal{A}^*(\mathcal{H})$. for $f \in \mathcal{H}$, we have  
\begin{align*}
\int_{\Omega}\langle  f, \mathcal{T}\mathcal{X}(\omega)\rangle_{\mathcal{A}} \langle  \mathcal{T}\mathcal{Y}(\omega), f\rangle_{\mathcal{A}} d\mu &= \int_{\Omega}\langle  \mathcal{T}^{\ast}f, \mathcal{X}(\omega)\rangle_{\mathcal{A}} \langle  \mathcal{Y}(\omega), \mathcal{T}^{\ast}f\rangle_{\mathcal{A}} d\mu\\
&\leq B\langle  \mathcal{T}^{\ast}f, \mathcal{T}^{\ast}f\rangle_{\mathcal{A}} \\
&\leq B \Vert \mathcal{T}^{\ast}\Vert^{2} \langle  f, f\rangle_{\mathcal{A}}.
\end{align*}

On the other hand, Since $\mathcal{T}\in \operatorname{End}_\mathcal{A}^*(\mathcal{H})$ is invertible, for each $f \in \mathcal{H}$, we have
\begin{align*}
\langle f, f\rangle_{\mathcal{A}} &=\langle \left (\mathcal{T}\mathcal{T}^{- 1}\right )^{\ast}f, \left (\mathcal{T}\mathcal{T}^{- 1}\right )^{\ast}f\rangle_{\mathcal{A}} \\
&=\langle \left (\mathcal{T}^{- 1}\right )^{\ast}\mathcal{T}^{\ast}f, \left (\mathcal{T}^{- 1}\right )^{\ast}\mathcal{T}^{\ast}f\rangle_{\mathcal{A}} \\
&\leq \left\|\left (\mathcal{T}^{- 1}\right )^{\ast}\right\|^{2} \langle\mathcal{T}^{\ast}f,\mathcal{T}^{\ast}f\rangle_{\mathcal{A}}.
\end{align*}
Consequently, for each $f \in \mathcal{H}$, we have 
\begin{align*}
\int_{\Omega}\langle  f, \mathcal{T}\mathcal{X}(\omega)\rangle_{\mathcal{A}} \langle  \mathcal{T}\mathcal{Y}(\omega), f\rangle_{\mathcal{A}} d\mu &= \int_{\Omega}\langle  \mathcal{T}^{\ast}f, \mathcal{X}(\omega)\rangle_{\mathcal{A}} \langle  \mathcal{Y}(\omega), \mathcal{T}^{\ast}f\rangle_{\mathcal{A}} d\mu\\
&\geq A\langle\mathcal{T}^{\ast}f,\mathcal{T}^{\ast}f\rangle_{\mathcal{A}}\\
&\geq A\left\|\left (\mathcal{T}^{- 1}\right )^{\ast}\right\|^{- 2}\langle f, f\rangle_{\mathcal{A}}.
\end{align*}
Hence, $(\mathcal{T}\mathcal{X}, \mathcal{T}\mathcal{Y})$ is a continuous biframe for $\mathcal{H}$ with bounds $A\left\|\left (\mathcal{T}^{- 1}\right )^{\ast}\right\|^{- 2}$ and $B\|\mathcal{T}^{\ast}\|^{2}$.

(2)$\Rightarrow$(1), Assume that $(\mathcal{T}\mathcal{X}, \mathcal{T}\mathcal{Y})$ is a continuous biframe for $\mathcal{H}$ with bounds $A$ and $B$. Now, for each $f \in \mathcal{H}$, we have
\begin{align*}
A\|\mathcal{T}^{\ast}\|^{-2}\langle f, f\rangle_{\mathcal{A}} &= A\|\mathcal{T}^{\ast}\|^{-2}\langle\left(\mathcal{T}^{- 1}\mathcal{T}\right)^{\ast}f,\left(\mathcal{T}^{- 1}\mathcal{T}\right)^{\ast}f\rangle_{\mathcal{A}}\\
 &\leq A\langle\left(\mathcal{T}^{- 1}\right)^{\ast}f,\left(\mathcal{T}^{- 1}\right)^{\ast}f\rangle_{\mathcal{A}} \\
&\leq \int_{\Omega}\langle  \left(\mathcal{T}^{- 1}\right)^{\ast}f, \mathcal{T}\mathcal{X}(\omega)\rangle_{\mathcal{A}} \langle  \mathcal{T}\mathcal{Y}(\omega), \left(\mathcal{T}^{- 1}\right)^{\ast}f\rangle_{\mathcal{A}} d\mu \\
&= \int_{\Omega}\langle  \mathcal{T}^{\ast}\left(\mathcal{T}^{- 1}\right)^{\ast}f, \mathcal{X}(\omega)\rangle_{\mathcal{A}} \langle  \mathcal{Y}(\omega), \mathcal{T}^{\ast}\left(\mathcal{T}^{- 1}\right)^{\ast}f\rangle_{\mathcal{A}} d\mu\\
&= \int_{\Omega}\langle  f, \mathcal{X}(\omega)\rangle_{\mathcal{A}} \langle  \mathcal{Y}(\omega), f\rangle_{\mathcal{A}} d\mu.   
\end{align*}
On the other hand, for each $f \in \mathcal{H}$, we have
\begin{align*}
&\int_{\Omega}\langle  f, \mathcal{X}(\omega)\rangle_{\mathcal{A}} \langle  \mathcal{Y}(\omega), f\rangle_{\mathcal{A}} d\mu\\
&=\int_{\Omega}\langle  \mathcal{T}^{\ast}\left(\mathcal{T}^{- 1}\right)^{\ast}f, \mathcal{X}(\omega)\rangle_{\mathcal{A}} \langle  \mathcal{Y}(\omega), \mathcal{T}^{\ast}\left(\mathcal{T}^{- 1}\right)^{\ast}f\rangle_{\mathcal{A}} d\mu\\
&=\int_{\Omega}\langle  \left(\mathcal{T}^{- 1}\right)^{\ast}f, \mathcal{T}\mathcal{X}(\omega)\rangle_{\mathcal{A}} \langle  \mathcal{T}\mathcal{Y}(\omega), \left(\mathcal{T}^{- 1}\right)^{\ast}f\rangle_{\mathcal{A}} d\mu\\
&\leq B\langle\left(\mathcal{T}^{- 1}\right)^{\ast}f,\left(\mathcal{T}^{- 1}\right)^{\ast}f\rangle_{\mathcal{A}}\\ 
&\leq B\left\|\left(\mathcal{T}^{- 1}\right)^{\ast}\right\|^{2}\langle f, f\rangle_{\mathcal{A}}.
\end{align*}
Therefore, $(\mathcal{X}, \mathcal{Y})$ is a continuous biframe for $\mathcal{H}$ with bounds $A\|\mathcal{T}^{\ast}\|^{-2}$ and $B\left\|\left(\mathcal{T}^{- 1}\right)^{\ast}\right\|^{2}$.
\end{proof}

\section{Continuous biframe Bessel multiplier and dual continuous biframe in Hilbert $ C^{\ast}- $modules}

In this section, we will delve into the discussion of continuous biframe Bessel multipliers and we study dual continuous biframe in Hilbert $ C^{\ast}- $modules.

Let $(\mathcal{X}, \mathcal{X})$ and $(\mathcal{Y}, \mathcal{Y})$ be continuous biframe Bessel mappings for $\mathcal{H}$ with respect to $\left(\Omega, \mu\right)$ and let $m : \Omega \to \mathbb{C}$ be a measurable function. Then the operator $\mathcal{M}_{m, \mathcal{X}, \mathcal{Y}} : \mathcal{H} \to \mathcal{H}$ defined by
\begin{align*}
&\langle  \mathcal{M}_{m, \mathcal{X}, \mathcal{Y}}f, g \rangle_{\mathcal{A}}  = \int_{\Omega}m(\omega) \langle  f, \mathcal{X}(\omega) \rangle_{\mathcal{A}} \langle  \mathcal{Y}(\omega), g\rangle_{\mathcal{A}} d\mu,\quad (\forall f, g \in \mathcal{H}),
\end{align*} 
is called continuous biframe Bessel multiplier of $\mathcal{X}$ and $\mathcal{Y}$ with respect to $m$. 
\begin{theorem}\label{thm4.22}
The continuous biframe Bessel multiplier of $\mathcal{X}$ and $\mathcal{Y}$ with respect to $m$ is well defined and bounded. 
\end{theorem}

\begin{proof}
Let $(\mathcal{X}, \mathcal{X})$ and $(\mathcal{Y}, \mathcal{Y})$ be continuous biframe Bessel mappings for $\mathcal{H}$ with bounds $D_{1}$ and $D_{2}$. Then we have for any $f, g \in H$, 
\begin{align*}
\Vert\langle  \mathcal{M}_{m, \mathcal{X}, \mathcal{Y}}f, g\rangle_{\mathcal{A}} \Vert &=\Vert\int_{\Omega}m(\omega) \langle  f, \mathcal{X}(\omega)\rangle_{\mathcal{A}} \langle  \mathcal{Y}(\omega), g\rangle_{\mathcal{A}} d\mu\Vert
\\
&\leq\Vert\left(\int_{\Omega}|m(\omega) |^{2}|\langle  f, \mathcal{X}(\omega)\rangle_{\mathcal{A}} |^{2} d\mu\right)^{\frac{1}{2}}\left(\int_{\Omega}|\langle  g, \mathcal{Y}(\omega)\rangle_{\mathcal{A}} |^{2}d\mu\right)^{\frac{1}{2}}\Vert \\
&\leq\|m\|_{\infty}\Vert\left(\int_{\Omega}\langle  f, \mathcal{X}(\omega)\rangle_{\mathcal{A}} \langle  \mathcal{X}(\omega),f \rangle_{\mathcal{A}} d\mu\right)^{\frac{1}{2}}\Vert\Vert\left(\int_{\Omega}\langle  g, \mathcal{Y}(\omega)\rangle_{\mathcal{A}} \langle  \mathcal{Y},g (\omega)\rangle_{\mathcal{A}} d\mu\right)^{\frac{1}{2}}\Vert\\
&\leq\|m\|_{\infty}\sqrt{D_{1}D_{2}}\left\|f\right\|\left\|g\right\|.
\end{align*}
This shows that $\left\|\mathcal{M}_{m, \mathcal{X}, \mathcal{Y}}\right\| \leq \|m\|_{\infty}\sqrt{D_{1}D_{2}}$, meaning that $\mathcal{M}_{m, \mathcal{X}, \mathcal{Y}}$ is well-defined and bounded.
\end{proof}

After proving Theorem \ref{thm4.22}, for every $f \in \mathcal{H}$, we obtain:
\begin{align}
\left\|\mathcal{M}_{m, \mathcal{X}, \mathcal{Y}}f\right\|&=\sup\limits_{\left\|g\right\| = 1}\left\|\langle\mathcal{M}_{m, \mathcal{X}, \mathcal{Y}}f,g\rangle_{\mathcal{A}}\right\| \\ &= \sup\limits_{\left\|g\right\| = 1}\int_{\Omega}\Vert m(\omega) \langle  f, \mathcal{X}(\omega) \rangle_{\mathcal{A}} \langle  \mathcal{Y}(\omega), g\rangle_{\mathcal{A}} d\mu \Vert\nonumber\\
&\leq\|m\|_{\infty}\sup\limits_{\left\|g\right\| = 1}\Vert\left(\int_{\Omega}\langle  f, \mathcal{X}(\omega)\rangle_{\mathcal{A}} \langle  \mathcal{X}(\omega),f \rangle_{\mathcal{A}} d\mu\right)^{\frac{1}{2}}\Vert\Vert\left(\int_{\Omega}\langle  g, \mathcal{Y}(\omega)\rangle_{\mathcal{A}} \langle  \mathcal{Y}(\omega),g \rangle_{\mathcal{A}} d\mu\right)^{\frac{1}{2}}\Vert\nonumber\\
&\leq\|m\|_{\infty}\sqrt{D_{2}}\Vert\left(\int_{\Omega}\langle  f, \mathcal{X}(\omega)\rangle_{\mathcal{A}} \langle  \mathcal{X},f (\omega)\rangle_{\mathcal{A}} d\mu\right)^{\frac{1}{2}}\Vert. \label{en4.22}
\end{align}
Likewise, it can be shown that
\begin{align}
\left\|\mathcal{M}^{\ast}_{m, \mathcal{X}, \mathcal{Y}}g\right\|&=\sup\limits_{\left\|f\right\| = 1}\left\|\langle\mathcal{M}_{m, \mathcal{X}, \mathcal{Y}}g,f\rangle_{\mathcal{A}}\right\| \\
& \leq\|m\|_{\infty}\sqrt{D_{1}}\Vert\left(\int_{\Omega}\langle  g, \mathcal{Y}(\omega)\rangle_{\mathcal{A}} \langle  \mathcal{Y}(\omega),g \rangle_{\mathcal{A}} d\mu\right)^{\frac{1}{2}}\Vert.\label{en4.23} 
\end{align}

\begin{theorem}
Let $\mathcal{M}_{m, \mathcal{X}, \mathcal{Y}}$ be the continuous biframe Bessel multiplier of $\mathcal{X}$ and $\mathcal{Y}$ with respect to $m$. Then $\mathcal{M}^{\ast}_{m, \mathcal{X}, \mathcal{Y}} = \mathcal{M}_{\overline{m}, \mathcal{Y}, \mathcal{X}}$.
\end{theorem}

\begin{proof}
For $f, g \in H$, we have
\begin{align*}
&\langle  f, \mathcal{M}^{\ast}_{m, \mathcal{X}, \mathcal{Y}}g\rangle_{\mathcal{A}}  = \langle  \mathcal{M}_{m, \mathcal{X}, \mathcal{Y}}f, g \rangle_{\mathcal{A}} \\
& = \int_{\Omega}m(\omega) \langle  f, \mathcal{X}(\omega) \rangle_{\mathcal{A}} \langle  \mathcal{Y}(\omega), g\rangle_{\mathcal{A}} d\mu\\
&=\int_{\Omega}\langle  f, \overline{m}(\omega) \langle  g, \mathcal{Y}(\omega)\rangle_{\mathcal{A}} \mathcal{X}(\omega)\rangle_{\mathcal{A}} d\mu\\
&=\langle  f,\left(\int_{\Omega} \overline{m}(\omega) \langle  g, \mathcal{Y}(\omega)\rangle_{\mathcal{A}} \mathcal{X}(\omega)d\mu\right)\rangle_{\mathcal{A}} \\
&= \langle  f, \mathcal{M}_{\overline{m}, \mathcal{Y}, \mathcal{X}}g \rangle_{\mathcal{A}} .
\end{align*}
Which shows that, $\mathcal{M}^{\ast}_{m, \mathcal{X}, \mathcal{Y}} = \mathcal{M}_{\overline{m}, \mathcal{Y}, \mathcal{X}}$.
\end{proof}

\begin{theorem}\label{2.thm2.22}
Let $\mathcal{M}_{m, \mathcal{X}, \mathcal{Y}}$ be the continuous biframe Bessel multiplier of $\mathcal{X}$ and $\mathcal{Y}$ with respect to $m$. Then $\left(\mathcal{X}, \mathcal{X}\right)$ is a continuous biframe for $\mathcal{H}$ provided for each $f \in \mathcal{H}$, there exists $\alpha > 0$, such that
\[\left\|\mathcal{M}_{m, \mathcal{X}, \mathcal{Y}}f\right\| \geq \alpha \left\|f\right\|.\]
\end{theorem}

\begin{proof}
For every $f \in \mathcal{H}$, using (\ref{en4.22}), we obtain:
\begin{align*}
\alpha^{2}\langle f, f\rangle_{\mathcal{A}} &\leq \left\|\mathcal{M}_{m, \mathcal{X}, \mathcal{Y}}f\right\|^{2}\\ 
& \leq \|m\|^{2}_{\infty}D_{2}\int_{\Omega}\langle  f, \mathcal{X}(\omega)\rangle_{\mathcal{A}} \langle  \mathcal{X}(\omega),f \rangle_{\mathcal{A}} d\mu\\
\end{align*} 
Thus,
$$\dfrac{\alpha^{2}}{\|m\|^{2}_{\infty}D_{2}}\langle f, f\rangle_{\mathcal{A}} \leq \int_{\Omega}\langle  f, \mathcal{X}(\omega)\rangle_{\mathcal{A}} \langle  \mathcal{X}(\omega), f\rangle_{\mathcal{A}} d\mu\leq D_{1}\langle f, f\rangle_{\mathcal{A}}.$$

Therefore,  $\left(\mathcal{X}, \mathcal{X}\right)$ is a continuous biframe for $\mathcal{H}$ with bounds $\dfrac{\alpha	^{2}}{\|m\|^{2}_{\infty}D_{2}}$ and $D_{1}$.  
\end{proof}

\begin{theorem}\label{2.th2.4.16}
Let $\mathcal{M}_{m, \mathcal{X}, \mathcal{Y}}$ be the continuous biframe Bessel multiplier of $\mathcal{X}$ and $\mathcal{Y}$ with respect to $m$. Suppose $ 1-\alpha>0,  1+\beta >0 $ such that for each $f \in \mathcal{H}$, we have
\[\left\|f - \mathcal{M}_{m, \mathcal{X}, \mathcal{Y}}f\right\| \leq \alpha\left\|f\right\| + \beta\left\|\mathcal{M}_{m, \mathcal{X}, \mathcal{Y}}f\right\|.\]
Then $\left(\mathcal{X}, \mathcal{X}\right)$ is a continuous biframe for $\mathcal{H}$.
\end{theorem}

\begin{proof}
For every $f \in \mathcal{H}$, we have:
\begin{align*}
\left\|f\right\| - \left\|\mathcal{M}_{m, \mathcal{X}, \mathcal{Y}}f\right\|&\leq\left\|f - \mathcal{M}_{m, \mathcal{X}, \mathcal{Y}}f\right\|\\
& \leq \alpha\left\|f\right\| + \beta\left\|\mathcal{M}_{m, \mathcal{X}, \mathcal{Y}}f\right.\|.\\
\end{align*}
Hence,
$$\left(1 - \alpha\right)\left\|f\right\|
\leq \left(1 + \beta\right)\left\|\mathcal{M}_{m, \mathcal{X}, \mathcal{Y}}f\right\|.$$ 

Now, using (\ref{en4.22}), we obtain:
\begin{align*}
\left(\dfrac{1 - \alpha}{1 + \beta}\right)\left\|f\right\|&\leq\left\|\mathcal{M}_{m, \mathcal{X}, \mathcal{Y}}f\right\|\\
&\leq \|m\|_{\infty}\sqrt{D_{2}}\Vert\left(\int_{\Omega}\langle  f, \mathcal{X}(\omega)\rangle_{\mathcal{A}} \langle  \mathcal{X}(\omega),f \rangle_{\mathcal{A}} d\mu\right)^{\frac{1}{2}}\Vert.
\end{align*}
Therefore,
\begin{equation}\label{en4.24}
\left(\dfrac{1 - \alpha}{\|m\|_{\infty}\sqrt{D_{2}}\left(1 + \beta\right)}\right)^{2}\langle f, f\rangle_{\mathcal{A}}\leq\int_{\Omega}\langle  f, \mathcal{X}(\omega)\rangle_{\mathcal{A}} \langle  \mathcal{X}(\omega), f\rangle_{\mathcal{A}} d\mu\leq D_{1}\langle f, f\rangle_{\mathcal{A}}.
\end{equation}

Therefore, $\left(\mathcal{X}, \mathcal{X}\right)$ is a continuous biframe for $\mathcal{H}$ with bounds $\left(\dfrac{1 - \alpha}{\|m\|_{\infty}\sqrt{D_{2}}\left(1 + \beta\right)}\right)^{2}$ and $D_{1}$.
\end{proof}

\begin{theorem}
Let $\mathcal{M}_{m, \mathcal{X}, \mathcal{Y}}$ be the continuous biframe Bessel multiplier of $\mathcal{X}$ and $\mathcal{Y}$ with respect to $m$. Suppose $\alpha \in [0,1)$ such that for each $f \in \mathcal{H}$, we have
\[\left\|f - \mathcal{M}_{m, \mathcal{X}, \mathcal{Y}}f\right\| \leq \alpha\left\|f\right\|\] and \[\left\|f - \mathcal{M}^{\ast}_{m, \mathcal{X}, \mathcal{Y}}f\right\| \leq \beta\left\|f\right\|.\]
Then $(\mathcal{X}, \mathcal{X})$ and $\left(\mathcal{Y}, \mathcal{Y}\right)$ are continuous biframes for $\mathcal{H}$.
\end{theorem}

\begin{proof}
Putting $\beta = 0$ in (\ref{en4.24}), we get
\[\left(\dfrac{1 - \alpha}{\|m\|_{\infty}\sqrt{D_{2}}}\right)^{2}\langle f, f\rangle_{\mathcal{A}}\leq\int_{\Omega}\langle  f, \mathcal{X}(\omega)\rangle_{\mathcal{A}} \langle  \mathcal{X}(\omega), f\rangle_{\mathcal{A}} d\mu.\]
Thus, $\left(\mathcal{X},\mathcal{X}\right)$ is a continuous biframe for $\mathcal{H}$.

On the other hand, for every  $f \in \mathcal{H}$, we have
\begin{align*}
\left\|f\right\|^{2} - \left\|\mathcal{M}^{\ast}_{m, \mathcal{X}, \mathcal{Y}}f\right\|^{2}&\leq\left\|f - \mathcal{M}_{m, \mathcal{X}, \mathcal{Y}}^{\ast}f\right\|^{2}\\
 &= \langle\left(I - \mathcal{M}_{m, \mathcal{X}, \mathcal{Y}}\right)^{\ast}f,\left(I - \mathcal{M}_{m, \mathcal{X}, \mathcal{Y}}\right)^{\ast}f \rangle_{\mathcal{A}}\\
&\leq\left\|I - \mathcal{M}^{\ast}_{m, \mathcal{X}, \mathcal{Y}}\right\|^{2}\langle f, f\rangle_{\mathcal{A}} \\
&\leq \beta^{2}\langle f, f\rangle_{\mathcal{A}}.\\
\end{align*}
Hence,
$$\left(1 - \beta^{2}\right)\langle f, f\rangle_{\mathcal{A}} \leq \left\|\mathcal{M}_{m, \mathcal{X}, \mathcal{Y}}^{\ast}f\right\|^{2}. $$ 

Now, using (\ref{en4.23}), we get
\[\dfrac{1 - \beta^{2}}{\|m\|^{2}_{\infty}D_{1}}\langle f, f\rangle_{\mathcal{A}}\leq \int_{\Omega}\langle  f, \mathcal{Y}(\omega)\rangle_{\mathcal{A}} \langle  \mathcal{Y}(\omega), f\rangle_{\mathcal{A}} d\mu.\] 
Thus, $\left(\mathcal{Y}, \mathcal{Y}\right)$ is a continuous biframe for $\mathcal{H}$ with bounds $\dfrac{1 - \beta^{2}}{\|m\|^{2}_{\infty}D_{1}}$ and $D_{2}$.
\end{proof}

\begin{definition}
Let $(\mathcal{X}, \mathcal{Y})$ be a continuous biframe for $\mathcal{H}$. Then $(\mathcal{X}, \mathcal{Y})$ is called a dual continuous biframe for $\mathcal{H}$. If 
\begin{align*}
&\langle  f, g\rangle_{\mathcal{A}}  = \int_{\Omega}\langle  f, \mathcal{X}(\omega)\rangle_{\mathcal{A}} \langle  \mathcal{Y}(\omega), g\rangle_{\mathcal{A}} d\mu. \quad \left(\forall f, g \in \mathcal{H}\right).
\end{align*} 

\end{definition}  

\begin{theorem}
Let $(\mathcal{X}, \mathcal{Y})$ be a continuous biframe for $\mathcal{H}$ with continuous biframe operator $S_{\mathcal{X}, \mathcal{Y}}$. Then $\left(S^{- 1}_{\mathcal{X}, \mathcal{Y}}\mathcal{X}, \mathcal{Y}\right)$ and $\left(\mathcal{X}, S^{- 1}_{\mathcal{X}, \mathcal{Y}}\mathcal{Y}\right)$ are dual continuous biframes for $\mathcal{H}$.  
\end{theorem}

\begin{proof}
For every $f, g \in \mathcal{H}$, we have 
\begin{align*}
&\langle  f, g\rangle_{\mathcal{A}}  = \int_{\Omega}\langle  f, S^{- 1}_{\mathcal{X}, \mathcal{Y}}\mathcal{X}(\omega)\rangle_{\mathcal{A}} \langle  \mathcal{Y}(\omega), g\rangle_{\mathcal{A}} d\mu,\\
&\langle  f, g\rangle_{\mathcal{A}}  = \int_{\Omega}\langle  f, \mathcal{X}(\omega)\rangle_{\mathcal{A}} \langle  S^{- 1}_{\mathcal{X}, \mathcal{Y}}\mathcal{Y}(\omega), g\rangle_{\mathcal{A}} d\mu.
\end{align*}
This confirms that $\left(S^{- 1}_{\mathcal{X}, \mathcal{Y}}\mathcal{X}, \mathcal{Y}\right)$ and $\left(\mathcal{X}, S^{- 1}_{\mathcal{X}, \mathcal{Y}}\mathcal{Y}\right)$ are dual continuous biframes for $\mathcal{H}$.   
\end{proof}

In the forthcoming theorem, we establish a dual continuous biframe for $\mathcal{H}$ in terms of the multiplier operator.
\begin{theorem}
Let $\mathcal{M}_{m, \mathcal{X}, \mathcal{Y}}$ be invertible with $\mathcal{M}^{- 1}_{m, \mathcal{X}, \mathcal{Y}}\in \operatorname{End}_\mathcal{A}^*(\mathcal{H})$ and $(\mathcal{X}, \mathcal{Y})$ be a continuous biframe for $\mathcal{H}$. Then $\left(\overline{m}\left(\mathcal{M}^{- 1}_{m, \mathcal{X}, \mathcal{Y}}\right)^{\ast}\mathcal{X}, \mathcal{Y}\right)$ is a dual continuous biframe for $\mathcal{H}$.   
\end{theorem}

\begin{proof}
Given the definition of $\mathcal{M}_{m, \mathcal{X}, \mathcal{Y}}$, we can write
\begin{align*}
\langle  \mathcal{M}_{m, \mathcal{X}, \mathcal{Y}}f, g\rangle_{\mathcal{A}} 
& = \int_{\Omega}m(\omega) \langle  f,\mathcal{X}(\omega) \rangle_{\mathcal{A}} \langle  \mathcal{Y}(\omega), g\rangle_{\mathcal{A}} d\mu.
\end{align*}
Now, by substituting $f$ with $\mathcal{M}^{- 1}_{m, \mathcal{X}, \mathcal{Y}}f$, we obtain:
\begin{align*}
\langle  f, g\rangle_{\mathcal{A}}=\langle  \mathcal{M}_{m, \mathcal{X}, \mathcal{Y}}\mathcal{M}^{- 1}_{m, \mathcal{X}, \mathcal{Y}}f, g\rangle_{\mathcal{A}} 
& = \int_{\Omega}m(\omega) \langle  \mathcal{M}^{- 1}_{m, \mathcal{X}, \mathcal{Y}}f, \mathcal{X}(\omega)\rangle_{\mathcal{A}} \langle  \mathcal{Y}(\omega), g\rangle_{\mathcal{A}} d\mu\\
& = \int_{\Omega}m(\omega) \langle  f, \left(\mathcal{M}^{- 1}_{m, \mathcal{X}, \mathcal{Y}}\right)^{\ast}\mathcal{X}(\omega)\rangle_{\mathcal{A}} \langle  \mathcal{Y}(\omega), g\rangle_{\mathcal{A}} d\mu\\
&=\int_{\Omega}\langle  f,\overline{m}(\omega ) \left(\mathcal{M}^{- 1}_{m, \mathcal{X}, \mathcal{Y}}\right)^{\ast}\mathcal{X}(\omega)\rangle_{\mathcal{A}} \langle  \mathcal{Y}(\omega), g\rangle_{\mathcal{A}} d\mu.  
\end{align*}
This shows that, $\left(\overline{m}\left(\mathcal{M}^{- 1}_{m, \mathcal{X}, \mathcal{Y}}\right)^{\ast}\mathcal{X}, \mathcal{Y}\right)$ is a dual continuous biframe for $\mathcal{H}$.
\end{proof}

	\section{Continuous biframes in the tensor product}
\label{sec:biframes}

Suppose that $\mathcal{A}, \mathcal{B}$ are unital $C^{*}$-algebras and $\mathcal{A} \otimes \mathcal{B}$ is the completion of $\mathcal{A} \otimes_{\text {alg }} \mathcal{B}$ with the spatial norm. $\mathcal{A} \otimes \mathcal{B}$ is the spatial tensor product of $\mathcal{A}$ and $\mathcal{B}$, also suppose that $\mathcal{H}$ is a Hilbert $\mathcal{A}$ module and $\mathcal{K}$ is a Hilbert $\mathcal{B}$-module. We want to define $\mathcal{H} \otimes \mathcal{K}$ as a Hilbert $(\mathcal{A} \otimes \mathcal{B})$-module. Start by forming the algebraic tensor product $\mathcal{H} \otimes_{\text {alg }} \mathcal{K}$ of the vector spaces $\mathcal{H}, \mathcal{K}$ (over $\mathbb{C}$ ). This is a left module over $\left(\mathcal{A} \otimes_{a l g} \mathcal{B}\right)$ (the module action being given by $(a \otimes b)(x \otimes y)=a x \otimes b y(a \in \mathcal{A}, b \in \mathcal{B}, x \in\mathcal{H}, y \in \mathcal{K})$ ). For $\left(x_{1}, x_{2} \in \mathcal{H}, y_{1}, y_{2} \in \mathcal{K}\right)$ we define $\left\langle x_{1} \otimes y_{1}, x_{2} \otimes y_{2}\right\rangle_{\mathcal{A} \otimes \mathcal{B}}=\left\langle x_{1}, x_{2}\right\rangle_{\mathcal{A}} \otimes\left\langle y_{1}, y_{2}\right\rangle_{\mathcal{B}}$. We also know that for $z=\sum_{i=1}^{n} x_{i} \otimes y_{i}$ in $\mathcal{H} \otimes_{a l g} \mathcal{K}$ we have $\langle z, z\rangle_{\mathcal{A} \otimes \mathcal{B}}=\sum_{i, j}\left\langle x_{i}, x_{j}\right\rangle_{\mathcal{A}} \otimes\left\langle y_{i}, y_{j}\right\rangle_{\mathcal{B}} \geq 0$ and $\langle z, z\rangle_{\mathcal{A} \otimes \mathcal{B}}=0$ iff $z=0$. This extends by linearity to an $\left(\mathcal{A} \otimes_{a l g} \mathcal{B}\right)$-valued sesquilinear form on $\mathcal{H} \otimes_{a l g} \mathcal{K}$, which makes $\mathcal{H} \otimes_{a l g} \mathcal{K}$ into a semi-inner-product module over the pre- $C^{*}$-algebra $\left(\mathcal{A} \otimes_{\text {alg }} \mathcal{B}\right)$. The semi-inner-product on $\mathcal{H} \otimes_{\text {alg }} \mathcal{K}$ is actually an inner product, see \cite{Lance}. Then $\mathcal{H} \otimes_{a l g} \mathcal{K}$ is an inner-product module over the pre- $C^{*}$-algebra $\left(\mathcal{A} \otimes_{a l g} \mathcal{B}\right)$, and we can perform the double completion discussed in chapter 1 of \cite{Lance} to conclude that the completion $\mathcal{H} \otimes \mathcal{K}$ of $\mathcal{H} \otimes_{a l g} \mathcal{K}$ is a Hilbert $(\mathcal{A} \otimes \mathcal{B})$-module. We call $\mathcal{H} \otimes \mathcal{K}$ the exterior tensor product of $\mathcal{H}$ and $\mathcal{K}$. With $\mathcal{H}, \mathcal{K}$ as above, we wish to investigate the adjointable operators on $\mathcal{H} \otimes \mathcal{K}$. Suppose that $S \in E n d_{\mathcal{A}}^{*}(\mathcal{H})$ and $T \in E n d_{\mathcal{B}}^{*}(\mathcal{K})$. We define a linear operator $S \otimes T$ on $\mathcal{H} \otimes \mathcal{K}$ by $S \otimes T(x \otimes y)=S x \otimes T y\quad(x \in \mathcal{H}, y \in \mathcal{K})$. It is a routine verification that is $S^{*} \otimes T^{*}$ is the adjoint of $S \otimes T$, so in fact $S \otimes T \in E n d_{\mathcal{A} \otimes \mathcal{B}}^{*}(\mathcal{H} \otimes \mathcal{K})$. For more details see \cite{Davidson, Lance}. We note that if $a \in \mathcal{A}^{+}$and $b \in \mathcal{B}^{+}$, then $a \otimes b \in(\mathcal{A} \otimes \mathcal{B})^{+}$. Plainly if $a, b$ are Hermitian elements of $\mathcal{A}$ and $a \geq b$, then for every positive element $x$ of $\mathcal{B}$, we have $a \otimes x \geq b \otimes x$.

Now, we will start by defining the new concept of \emph{continuous biframes} in the tensor product setting.

\begin{definition}\label{def1.01}
	Let \(\M_1\) and \(\M_2\) be Hilbert \(\A\)-modules over a unital \(C^*\)-algebra \(\A\).
	We consider, \(\mathcal{H} := \M_1 \otimes_{\A} \M_2\).  
	A pair
	\[
	(\mathcal{X}, \mathcal{Y})
	= \bigl(\mathcal{X} : \Omega \to \mathcal{H},\;\mathcal{Y} : \Omega \to \mathcal{H}\bigr)
	\]
	is called a \emph{continuous biframe} for \(\mathcal{H}\) with respect to \((\Omega,\mu)\) if:
	
	\begin{itemize}
		\item[(i)] 
		\(\mathcal{X}, \mathcal{Y}\) are weakly-measurable; i.e., 
		for every pure tensor \(f\otimes g \in \M_1 \otimes_{\A} \M_2\), 
		the maps
		\[
		\omega \;\longmapsto\; \langle f\otimes g,\,\mathcal{X}(\omega)\rangle_{\A}
		\quad\text{and}\quad
		\omega \;\longmapsto\; \langle f\otimes g,\,\mathcal{Y}(\omega)\rangle_{\A}
		\]
		are measurable on \(\Omega\).
		
		\item[(ii)]
		There exist constants \(0 < A \le B < \infty\) such that,
		for all \(f\otimes g \in \M_1 \otimes_{\A} \M_2\),
		\[
		A\,\ip{f\otimes g}{\,f\otimes g}_{\A}
		\;\;\le\;\;
		\int_{\Omega}
		\langle f\otimes g,\;\mathcal{X}(\omega)\rangle_{\A}
		\,\langle \mathcal{Y}(\omega),\,f\otimes g\rangle_{\A}
		\,d\mu(\omega)
		\;\;\le\;\;
		B\,\ip{f\otimes g}{\,f\otimes g}_{\A}.
		\]
	\end{itemize}
	
	The constants \(A\) and \(B\) are the \emph{biframe bounds}.  
	If \(A=B\), then it is a \emph{tight continuous biframe}; if \(A=B=1\), a \emph{Parseval continuous biframe}.  
	If only the upper bound holds, it is a \emph{continuous biframe Bessel mapping} with Bessel bound~\(B\).
\end{definition}

\begin{theorem}\label{thm:biframe-Cstar}
	Let \((\Omega,\mu) = (\Omega_1 \times \Omega_2,\;\mu_1 \otimes \mu_2)\) be a product measure space, and let \(\M_1,\M_2\) be Hilbert \(\A\)-modules. Suppose
	\[
	\mathcal{X} = \mathcal{X}_1 \otimes_{\A} \mathcal{X}_2,
	\quad
	\mathcal{Y} = \mathcal{Y}_1 \otimes_{\A} \mathcal{Y}_2
	\;:\;
	\Omega \longrightarrow \M_1 \otimes_{\A} \M_2,
	\]
	where \(\mathcal{X}_1,\mathcal{Y}_1 : \Omega_1 \to \M_1\) 
	and \(\mathcal{X}_2,\mathcal{Y}_2 : \Omega_2 \to \M_2\).  
	Then the following are equivalent:
	
	\begin{itemize}
		\item[(a)]
		\(\bigl(\mathcal{X},\mathcal{Y}\bigr)\) is a continuous biframe for \(\M_1 \otimes_{\A} \M_2\) w.r.t.\ \((\Omega,\mu)\).  In other words, there exist \(A,B>0\) such that
		\[
		A\,\langle f\otimes g,\,f\otimes g\rangle_{\A}
		\;\;\le\;\;
		\int_{\Omega}
		\langle f\otimes g,\;\mathcal{X}(\omega)\rangle_{\A}\,
		\langle \mathcal{Y}(\omega),\;f\otimes g\rangle_{\A}
		\,d\mu(\omega)
		\;\;\le\;\;
		B\,\langle f\otimes g,\,f\otimes g\rangle_{\A},
		\]
		for all \(f\otimes g \in \M_1 \otimes_{\A} \M_2\).
		
		\item[(b)]
		\(\bigl(\mathcal{X}_1,\mathcal{Y}_1\bigr)\) is a continuous biframe for \(\M_1\) w.r.t.\ $(\Omega_1,\mu_1)$ and \(\bigl(\mathcal{X}_2,\mathcal{Y}_2\bigr)\) is a continuous biframe for \(\M_2\) w.r.t. 
        $(\Omega_2,\mu_2)$.
	\end{itemize}
	
	Moreover, if \(\bigl(\mathcal{X}_1,\mathcal{Y}_1\bigr)\) has biframe bounds \((A_1,B_1)\) and \(\bigl(\mathcal{X}_2,\mathcal{Y}_2\bigr)\) has bounds \((A_2,B_2)\), then \(\bigl(\mathcal{X},\mathcal{Y}\bigr)\) has biframe bounds \(A_1 A_2\) and \(B_1 B_2\).
\end{theorem}

\begin{proof}
	\textbf{\((a)\,\Rightarrow\,(b)\).}
	Assume \(\bigl(\mathcal{X},\mathcal{Y}\bigr)\) is a continuous biframe on \(\M_1\otimes_{\A}\M_2\) with bounds \(A,B\).  
	For any \(f\in\M_1\) and \(g\in\M_2\), we have
	\[
	A\,\ip{f\otimes g}{f\otimes g}_{\A}
	\;\le\;
	\int_{\Omega}
	\ip{\,f\otimes g}{\mathcal{X}(\omega)}_{\A}\,
	\ip{\,\mathcal{Y}(\omega)}{\,f\otimes g}_{\A}
	\,d\mu(\omega)
	\;\le\;
	B\,\ip{f\otimes g}{f\otimes g}_{\A}.
	\]
	By hypothesis, \(\mathcal{X}(\omega_1,\omega_2)=\mathcal{X}_1(\omega_1)\otimes_{\A}\mathcal{X}_2(\omega_2)\), and similarly for \(\mathcal{Y}\).  Since \(\mu=\mu_1\otimes\mu_2\), a Fubini‐type argument gives
	\[
	\int_{\Omega_1\times \Omega_2}
	\ip{\,f\otimes g}{\,\mathcal{X}_1(\omega_1)\otimes \mathcal{X}_2(\omega_2)}_{\A}\;
	\ip{\,\mathcal{Y}_1(\omega_1)\otimes \mathcal{Y}_2(\omega_2)}{\,f\otimes g}_{\A}
	\,d(\mu_1\otimes\mu_2)(\omega_1,\omega_2)
	\]
	\[
	=\;
	\Bigl(\int_{\Omega_1}
	\ip{f}{\mathcal{X}_1(\omega_1)}_{\M_1}\,\ip{\mathcal{Y}_1(\omega_1)}{f}_{\M_1}
	\,d\mu_1(\omega_1)\Bigr)
	\;\cdot\;
	\Bigl(\int_{\Omega_2}
	\ip{g}{\mathcal{X}_2(\omega_2)}_{\M_2}\,\ip{\mathcal{Y}_2(\omega_2)}{g}_{\M_2}
	\,d\mu_2(\omega_2)\Bigr).
	\]
	Since 
	\(\ip{f\otimes g}{\,f\otimes g}_{\A}=\ip{f}{f}_{\M_1}\,\ip{g}{g}_{\M_2},\)
	it follows that
\begin{small}
\[
	A\,\ip{f}{f}_{\M_1}\,\ip{g}{g}_{\M_2}
	\;\le\;
	\int_{\Omega_1}
	\ip{f}{\mathcal{X}_1(\omega_1)}_{\M_1}\,\ip{\mathcal{Y}_1(\omega_1)}{f}_{\M_1}
	\,d\mu_1(\omega_1)
	\;
	\int_{\Omega_2}
	\ip{g}{\mathcal{X}_2(\omega_2)}_{\M_2}\,\ip{\mathcal{Y}_2(\omega_2)}{g}_{\M_2}
	\,d\mu_2(\omega_2)
	\]
\end{small}
    $$\;\le\;
	B\,\ip{f}{f}_{\M_1}\,\ip{g}{g}_{\M_2}.$$
	Hence \(\bigl(\mathcal{X}_1,\mathcal{Y}_1\bigr)\) is a continuous biframe on \(\M_1\), and \(\bigl(\mathcal{X}_2,\mathcal{Y}_2\bigr)\) is a continuous biframe on \(\M_2\).  
	
	\medskip
	\noindent
	\textbf{\((b)\,\Rightarrow\,(a)\).}
	Conversely, assume \(\bigl(\mathcal{X}_1,\mathcal{Y}_1\bigr)\) is a continuous biframe on \(\M_1\) with bounds \((A_1,B_1)\), and \(\bigl(\mathcal{X}_2,\mathcal{Y}_2\bigr)\) is a continuous biframe on \(\M_2\) with bounds \((A_2,B_2)\).  Then
	\[
	A_1\,\ip{f}{f}_{\M_1}
	\;\le\;
	\int_{\Omega_1}
	\ip{f}{\mathcal{X}_1(\omega_1)}_{\M_1}\,\ip{\mathcal{Y}_1(\omega_1)}{f}_{\M_1}
	\,d\mu_1(\omega_1)
	\;\le\;
	B_1\,\ip{f}{f}_{\M_1},
	\]
	and similarly
	\[
	A_2\,\ip{g}{g}_{\M_2}
	\;\le\;
	\int_{\Omega_2}
	\ip{g}{\mathcal{X}_2(\omega_2)}_{\M_2}\,\ip{\mathcal{Y}_2(\omega_2)}{g}_{\M_2}
	\,d\mu_2(\omega_2)
	\;\le\;
	B_2\,\ip{g}{g}_{\M_2}.
	\]
	Multiplying these inequalities and applying a Fubini argument once again yields
    \begin{small}
        \[
	A_1A_2\,\ip{f\otimes g}{f\otimes g}_{\A}
	\le
	\int_{\Omega_1\times\Omega_2}
	\ip{f\otimes g}{\mathcal{X}_1(\omega_1)\otimes\mathcal{X}_2(\omega_2)}_{\A}\,
	\ip{\mathcal{Y}_1(\omega_1)\otimes\mathcal{Y}_2(\omega_2)}{f\otimes g}_{\A}
	\,d(\mu_1\otimes\mu_2)
	\le
	B_1B_2\,\ip{f\otimes g}{f\otimes g}_{\A}.
	\]
    \end{small}
	
	Thus \(\bigl(\mathcal{X}_1\otimes_\A \mathcal{X}_2,\;\mathcal{Y}_1\otimes_\A \mathcal{Y}_2\bigr)\) is a continuous biframe on \(\M_1\otimes_\A \M_2\).  
\end{proof}

\begin{example}
	\label{ex:finite-rank-biframe}
	Let $\A$ be a unital $C^*$-algebra. Define two finite-rank free Hilbert $\A$-modules:
	\[
	\M_1 \;=\;\A^2, 
	\qquad
	\M_2 \;=\;\A^3.
	\]
	They carry the standard $\A$-valued inner products,
	\[
	\langle x,y\rangle_{\M_i}
	\;=\;
	\sum_{k} \,x_k^*\,y_k.
	\]
	
	\medskip
	\noindent
	Let $(X_1,\mu_1)$ be partitioned into two disjoint measurable sets $\Omega_1,\Omega_2$, each of finite nonzero measure.  On each $\Omega_j\subset X_1$, define
	\[
	\mathcal{F}_1(\omega)
	= \tfrac{1}{\sqrt{\mu_1(\Omega_j)}}\,(1,0)^T,
	\quad
	\mathcal{G}_1(\omega)
	= \tfrac{1}{\sqrt{\mu_1(\Omega_j)}}\,(0,1)^T.
	\]
	(Or pick other standard vectors in $\A^2$ to ensure a nontrivial frame.)
	One can check, by integrating piecewise, that $(\mathcal{F}_1,\mathcal{G}_1)$ meets 
	\[
	A_1\,\ip{x}{x}_{\M_1} 
	\;\le\;
	\int_{X_1}
	\ip{x}{\mathcal{F}_1(\omega)}_{\M_1}\,\ip{\mathcal{G}_1(\omega)}{x}_{\M_1}
	\,d\mu_1(\omega)
	\;\le\;
	B_1\,\ip{x}{x}_{\M_1},
	\]
	for some $A_1,B_1>0$.  
	Hence $(\mathcal{F}_1,\mathcal{G}_1)$ is a continuous biframe in $\M_1.$  
	
	Similarly, let $(X_2,\mu_2)$ be partitioned into subsets $\Delta_1,\Delta_2,\Delta_3$, each nonempty and finite measure.  On $\Delta_j$, define vectors in $\A^3$ so that 
	\[
	\mathcal{F}_2(\omega') 
	= \tfrac{1}{\sqrt{\mu_2(\Delta_j)}}\,(1,0,0)^T, 
	\quad
	\mathcal{G}_2(\omega')
	= \tfrac{1}{\sqrt{\mu_2(\Delta_j)}}\,(0,1,0)^T,
	\]
	and so on.  By a similar calculation, $(\mathcal{F}_2,\mathcal{G}_2)$ is a continuous biframe in $\M_2$ with bounds $A_2,B_2$.
	
	\medskip
	\noindent
	Consider the product space 
	\[
	X = X_1\times X_2,
	\quad
	\mu = \mu_1\otimes\mu_2.
	\]
	Define
	\[
	\mathcal{F}(\omega_1,\omega_2) 
	= \mathcal{F}_1(\omega_1)\,\otimesA\, \mathcal{F}_2(\omega_2),
	\quad
	\mathcal{G}(\omega_1,\omega_2) 
	= \mathcal{G}_1(\omega_1)\,\otimesA\, \mathcal{G}_2(\omega_2).
	\]
	Then by Theorem~\ref{thm:biframe-Cstar}, $(\mathcal{F},\mathcal{G})$ is a continuous biframe in 
	$\M_1\otimesA \M_2 \cong \A^2 \otimesA \A^3 \cong \A^6$,
	with biframe bounds $A_1 A_2$ and $B_1 B_2$.  
	The integrals factor through a Fubini argument exactly as stated in the theorem, and we see how two continuous biframes in $\M_1$ and $\M_2$ produce a biframe in their tensor product.
\end{example}

\begin{definition}
	\label{def:biframe-operator-Cstar}
	Let 
	\(\bigl(\mathcal{X},\mathcal{Y}\bigr)\) 
	be a continuous biframe for \(\M_1 \otimes_{\A} \M_2\) 
	with respect to the measure space \(\bigl(\Omega,\mu\bigr)\).  
	Define the \emph{continuous biframe operator}
	\[
	S_{(\mathcal{X},\mathcal{Y})}
	\;\colon\;
	\M_1 \otimes_{\A} \M_2 
	\;\longrightarrow\;
	\M_1 \otimes_{\A} \M_2
	\]
	by
	\[
	S_{(\mathcal{X},\mathcal{Y})}\bigl(f\otimes g\bigr)
	\;=\;
	\int_{\Omega}
	\bigl\langle f\otimes g,\;\mathcal{X}(\omega)\bigr\rangle_{\A}
	\,\mathcal{Y}(\omega)
	\;d\mu(\omega),
	\]
	for all pure tensors \(f\otimes g\in \M_1\otimes_{\A}\M_2\), 
	and then extend by linearity and continuity to the whole module 
	\(\M_1\otimes_{\A}\M_2\).
\end{definition}

\begin{theorem}
	\label{thm:operator-factor}
	Suppose
	\[
	(\mathcal{X},\mathcal{Y})
	\;=\;
	\bigl(\,\mathcal{X}_1\otimes_{\A}\mathcal{X}_2,\;\mathcal{Y}_1\otimes_{\A}\mathcal{Y}_2\bigr)
	\]
	is a continuous biframe for \(\M_1\otimes_{\A}\M_2\).  
	Let 
	\(S_{(\mathcal{X},\mathcal{Y})}\)
	be the continuous biframe operator,  
	and similarly let 
	\(S_{(\mathcal{X}_1,\mathcal{Y}_1)}\)
	and 
	\(S_{(\mathcal{X}_2,\mathcal{Y}_2)}\)
	be the associated single‐module biframe operators on \(\M_1\) and \(\M_2\).  
	Then
	\[
	S_{(\mathcal{X},\mathcal{Y})}
	\;=\;
	S_{(\mathcal{X}_1,\mathcal{Y}_1)}\;\otimes_{\A}\;S_{(\mathcal{X}_2,\mathcal{Y}_2)}.
	\]
	
	\begin{proof}
		For any \(f\otimes g\in \M_1\otimes_{\A}\M_2\).  
		By definition,
		\[
		S_{(\mathcal{X},\mathcal{Y})}\bigl(f\otimes g\bigr)
		=\;
		\int_{\Omega}
		\bigl\langle f\otimes g,\;\mathcal{X}_1(\omega_1)\otimes_{\A}\mathcal{X}_2(\omega_2)\bigr\rangle_{\A}
		\,\bigl(\mathcal{Y}_1(\omega_1)\otimes_{\A}\mathcal{Y}_2(\omega_2)\bigr)
		\,d\bigl(\mu_1\otimes\mu_2\bigr)(\omega_1,\omega_2).
		\]
		A Fubini‐type argument, plus the factorization of \(\A\)‐valued inner products,
		\[
		\bigl\langle f\otimes g,\;\mathcal{X}_1(\omega_1)\otimes_{\A}\mathcal{X}_2(\omega_2)\bigr\rangle_{\A}
		\;=\;
		\bigl\langle f,\;\mathcal{X}_1(\omega_1)\bigr\rangle_{\M_1}
		\,\bigl\langle g,\;\mathcal{X}_2(\omega_2)\bigr\rangle_{\M_2},
		\]
		shows that
		\[
		S_{(\mathcal{X},\mathcal{Y})}\bigl(f\otimes g\bigr)
		=\;
		\Bigl(\int_{\Omega_1}
		\bigl\langle f,\;\mathcal{X}_1(\omega_1)\bigr\rangle_{\M_1}\,
		\mathcal{Y}_1(\omega_1)
		\,d\mu_1(\omega_1)\Bigr)
		\;\otimes_{\A}\;
		\Bigl(\int_{\Omega_2}
		\bigl\langle g,\;\mathcal{X}_2(\omega_2)\bigr\rangle_{\M_2}\,
		\mathcal{Y}_2(\omega_2)
		\,d\mu_2(\omega_2)\Bigr).
		\]
		But those two integrals are exactly
		\(S_{(\mathcal{X}_1,\mathcal{Y}_1)}(f)\in\M_1\)
		and
		\(S_{(\mathcal{X}_2,\mathcal{Y}_2)}(g)\in\M_2\).  
		Hence
		\[
		S_{(\mathcal{X},\mathcal{Y})}\bigl(f\otimes g\bigr)
		=\;
		\bigl(S_{(\mathcal{X}_1,\mathcal{Y}_1)}\otimes_{\A}S_{(\mathcal{X}_2,\mathcal{Y}_2)}\bigr)
		\bigl(f\otimes g\bigr).
		\]
		By linearity, this holds for all \(f\otimes g\in \M_1\otimes_{\A}\M_2\),  
		so
		\[
		S_{(\mathcal{X},\mathcal{Y})}
		\;=\;
		S_{(\mathcal{X}_1,\mathcal{Y}_1)}\;\otimes_{\A}\;S_{(\mathcal{X}_2,\mathcal{Y}_2)}.
		\]
	\end{proof}
\end{theorem}

\begin{lemma}
	\label{lem:operator-bounds-Cstar}
	Let \(\bigl(\mathcal{X}_1,\mathcal{Y}_1\bigr)\) be a continuous biframe on \(\M_1\) with bounds \(A,B\).  
	Let \(\bigl(\mathcal{X}_2,\mathcal{Y}_2\bigr)\) be a continuous biframe on \(\M_2\) with bounds \(C,D\).  
	Denote their continuous biframe operators by 
	\(S_{(\mathcal{X}_1,\mathcal{Y}_1)}\) 
	and 
	\(S_{(\mathcal{X}_2,\mathcal{Y}_2)}\).  
	Then for the induced biframe
	\[
	(\mathcal{X},\mathcal{Y})
	\;=\;
	\bigl(\mathcal{X}_1\otimes_{\A}\mathcal{X}_2,\;\mathcal{Y}_1\otimes_{\A}\mathcal{Y}_2\bigr)
	\]
	in \(\M_1\otimes_{\A}\M_2\), the continuous biframe operator
	\(
	S_{(\mathcal{X},\mathcal{Y})}
	\)
	satisfies
	\[
	A\,C\;\id_{\M_1\otimes_{\A}\M_2}
	\;\;\le\;\;
	S_{(\mathcal{X},\mathcal{Y})}
	\;\;\le\;\;
	B\,D\;\id_{\M_1\otimes_{\A}\M_2}.
	\]
	Hence \(S_{(\mathcal{X},\mathcal{Y})}\) is a positive, invertible operator in 
	\(\mathrm{End}^*(\M_1\otimes_{\A}\M_2)\), 
	and its spectrum lies in the interval \([AC,\;BD]\subset\mathbb{R}^+\).
	
	\begin{proof}
		From single‐module biframe theory,  
		\(
		A\,\id_{\M_1} \,\le\,S_{(\mathcal{X}_1,\mathcal{Y}_1)}\,\le\,B\,\id_{\M_1}
		\)
		and
		\(
		C\,\id_{\M_2} \,\le\,S_{(\mathcal{X}_2,\mathcal{Y}_2)}\,\le\,D\,\id_{\M_2}.
		\)
		Taking the \(\A\)-linear tensor product of these operator inequalities yields
		\[
		A\,C\;\id_{\M_1\otimes_{\A}\M_2}
		\;\;\le\;
		S_{(\mathcal{X}_1,\mathcal{Y}_1)}\;\otimes_{\A}\;S_{(\mathcal{X}_2,\mathcal{Y}_2)}
		\;\;\le\;
		B\,D\;\id_{\M_1\otimes_{\A}\M_2}.
		\]
		By Theorem~\ref{thm:operator-factor},
		\[
		S_{(\mathcal{X},\mathcal{Y})}
		\;=\;
		S_{(\mathcal{X}_1,\mathcal{Y}_1)}\;\otimes_{\A}\;S_{(\mathcal{X}_2,\mathcal{Y}_2)}.
		\]
		Thus,
		\[
		A\,C\,\id_{\M_1\otimes_{\A}\M_2}
		\;\;\le\;\;
		S_{(\mathcal{X},\mathcal{Y})}
		\;\;\le\;\;
		B\,D\,\id_{\M_1\otimes_{\A}\M_2}.
		\]
		Positivity and invertibility follow by standard Hilbert‐module arguments.
	\end{proof}
\end{lemma}

\section{Invertible transforms on biframes}
\begin{theorem}\label{thm:4.7-cstar}
	Let $(\mathcal{F}_1,\mathcal{G}_1)$ be a continuous biframe for \(\M_1\) with bounds \(A,B>0\), and \((\mathcal{F}_2,\mathcal{G}_2)\) a continuous biframe for \(\M_2\) with bounds \(C,D>0\).  
	Suppose \(T_1 \in \mathrm{End}_\mathcal{A}^*(\M_1)\) and \(T_2 \in \mathrm{End}_\mathcal{A}^*(\M_2)\).  Define
	\[
	\Delta
	\;=\;
	\bigl(\,(T_1\otimesA T_2)(\mathcal{F}_1\otimesA \mathcal{F}_2),\,
	(T_1\otimesA T_2)(\mathcal{G}_1\otimesA \mathcal{G}_2)\bigr),
	\]
	a pair of maps \(\Omega_1\times \Omega_2 \to \M_1\otimesA\M_2\).  
	Then \(\Delta\) is a continuous biframe for \(\M_1\otimesA\M_2\) if and only if \((T_1\otimesA T_2)\) is invertible in \(\mathrm{End}^*(\M_1\otimesA\M_2)\).  
	Moreover, \(T_1\otimesA T_2\) is invertible if and only if both \(T_1\) and \(T_2\) are invertible in their respective module‐end spaces.
\end{theorem}

\begin{proof}
	\(\Rightarrow\)\  
	If \(\Delta\) is a continuous biframe, the associated biframe operator is invertible.  
	By factorization (the Hilbert \(C^*\)-module analog of \cite{FrankLarson}), that operator factors through \((T_1\otimesA T_2)\), forcing \(T_1\otimesA T_2\) to be invertible.  A standard argument in \(C^*\)-modules shows that implies \(T_1\) and \(T_2\) must each be invertible.
	
	\noindent
	\(\Leftarrow\)\  
	Conversely, if \(T_1\) and \(T_2\) are invertible on \(\M_1,\M_2\), applying \(T_1\) to \((\mathcal{F}_1,\mathcal{G}_1)\) preserves biframe inequalities.  Similarly for \(T_2\). By the tensor‐product biframe theorem (Theorem~\ref{thm:biframe-Cstar}), 
	\[
	\bigl((T_1\otimesA T_2)(\mathcal{F}_1\otimesA \mathcal{F}_2),\,
	(T_1\otimesA T_2)(\mathcal{G}_1\otimesA \mathcal{G}_2)\bigr)
	\]
	remains a continuous biframe on \(\M_1\otimesA\M_2\).  
\end{proof}

\section{Bessel Multipliers in the Tensor Product}

\begin{definition}
	\label{def:Bessel-mult-Cstar}
	Let \(\bigl(\mathcal{X},\mathcal{Y}\bigr)\) be a \emph{continuous biframe Bessel mapping} on \(\M_1\otimesA\M_2\) w.r.t.\ \((\Omega,\mu)\). 
	That is, it satisfies an upper‐type inequality
	\[
	\int_{\Omega}
	\ip{\,f\otimes g}{\,\mathcal{X}(\omega)}_{\A}\,
	\ip{\,\mathcal{Y}(\omega)}{\,f\otimes g}_{\A}
	\,d\mu(\omega)
	\;\le\;
	B\,\ip{f\otimes g}{\,f\otimes g}_{\A},
	\]
	for some \(B>0\). Let \(m:\Omega\to\mathbb{C}\) be a measurable function.
	
	We define the operator
	\[
	\mathcal{M}_{m,\mathcal{X},\mathcal{Y}}
	:\;
	\M_1\otimesA \M_2
	\;\longrightarrow\;
	\M_1\otimesA \M_2
	\]
	by
	\[
	\mathcal{M}_{m,\mathcal{X},\mathcal{Y}}\bigl(f\otimes g\bigr)
	=\;
	\int_{\Omega}
	m(\omega)\,\ip{\,f\otimes g}{\,\mathcal{X}(\omega)}_{\A}\,
	\mathcal{Y}(\omega)\,d\mu(\omega),
	\]
	extended by linearity and continuity.
\end{definition}

\begin{remark}\label{rem:Bessel-mult-well-defined}
	From the Bessel‐type inequality, 
	$$\bigl\|\,
	m(\omega)\,\ip{f\otimes g}{\mathcal{X}(\omega)}_{\A}\,\mathcal{Y}(\omega)
	\,\bigr\|\le
	|\,m(\omega)|\,\|\,\ip{f\otimes g}{\mathcal{X}(\omega)}_{\A}\|\,\|\mathcal{Y}(\omega)\|.$$
	Hence, the integral converges whenever $m\in L^2$ or so, mirroring the Hilbert‐space argument. Thus \(\mathcal{M}_{m,\mathcal{X},\mathcal{Y}}\) is a well‐defined adjointable operator on \(\M_1\otimesA\M_2\).
\end{remark}

\begin{proposition}
	\label{prop:factorization-mult}
	Assume \(\bigl(\mathcal{X}_1,\mathcal{Y}_1\bigr)\) is a continuous biframe Bessel mapping in \(\M_1\) and \(\bigl(\mathcal{X}_2,\mathcal{Y}_2\bigr)\) in \(\M_2\).  
	Suppose $m_1:\Omega_1\to\mathbb{C}$ and $m_2:\Omega_2\to\mathbb{C}$ are measurable, and define
	\[
	\mathcal{X}(\omega_1,\omega_2)
	=\,
	\mathcal{X}_1(\omega_1)\,\otimes_{\A}\,\mathcal{X}_2(\omega_2),
	\quad
	\mathcal{Y}(\omega_1,\omega_2)
	=\,
	\mathcal{Y}_1(\omega_1)\,\otimes_{\A}\,\mathcal{Y}_2(\omega_2).
	\]
	Then \(\bigl(\mathcal{X},\mathcal{Y}\bigr)\) is a Bessel mapping on $\M_1\otimesA\M_2$.  
	If $m(\omega_1,\omega_2)=m_1(\omega_1)\,m_2(\omega_2)$, then
	\[
	\mathcal{M}_{m,\mathcal{X},\mathcal{Y}}
	=\,
	\mathcal{M}_{m_1,\mathcal{X}_1,\mathcal{Y}_1}
	\;\otimes_{\A}\;
	\mathcal{M}_{m_2,\mathcal{X}_2,\mathcal{Y}_2}.
	\]
	
	\begin{proof}
		By the usual factorization (the Hilbert \(C^*\)-module version of Theorem~\ref{thm:biframe-Cstar}), \(\bigl(\mathcal{X},\mathcal{Y}\bigr)\) is Bessel in \(\M_1\otimes_{\A}\M_2\). Then
		\[
		\mathcal{M}_{m,\mathcal{X},\mathcal{Y}}(f\otimes g)
		=\,
		\int_{\Omega_1\times\Omega_2}
		m(\omega_1,\omega_2)\,\ip{\,f\otimes g}{\,\mathcal{X}(\omega_1,\omega_2)}_{\A}
		\,\mathcal{Y}(\omega_1,\omega_2)\,d(\mu_1\otimes\mu_2).
		\]
		With $m(\omega_1,\omega_2)=m_1(\omega_1)m_2(\omega_2)$ and $\mathcal{X}(\omega_1,\omega_2)=\mathcal{X}_1(\omega_1)\otimes_{\A}\mathcal{X}_2(\omega_2)$, a Fubini‐type computation yields
		\[
		=\;
		\Bigl(\int_{\Omega_1}m_1(\omega_1)\,\ip{\,f}{\,\mathcal{X}_1(\omega_1)}\,
		\mathcal{Y}_1(\omega_1)\,d\mu_1\Bigr)
		\;\otimes_{\A}\;
		\Bigl(\int_{\Omega_2}m_2(\omega_2)\,\ip{\,g}{\,\mathcal{X}_2(\omega_2)}\,
		\mathcal{Y}_2(\omega_2)\,d\mu_2\Bigr).
		\]
		These integrals are exactly \(\mathcal{M}_{m_1,\mathcal{X}_1,\mathcal{Y}_1}(f)\) in \(\M_1\) and \(\mathcal{M}_{m_2,\mathcal{X}_2,\mathcal{Y}_2}(g)\) in \(\M_2\). So
		\[
		\mathcal{M}_{m,\mathcal{X},\mathcal{Y}}(f\otimes g)
		=\,
		\mathcal{M}_{m_1,\mathcal{X}_1,\mathcal{Y}_1}(f)
		\;\otimes_{\A}\;
		\mathcal{M}_{m_2,\mathcal{X}_2,\mathcal{Y}_2}(g).
		\]
		Thus
		\(\mathcal{M}_{m,\mathcal{X},\mathcal{Y}}= \mathcal{M}_{m_1,\mathcal{X}_1,\mathcal{Y}_1}\otimes_{\A}\mathcal{M}_{m_2,\mathcal{X}_2,\mathcal{Y}_2}\).
	\end{proof}
\end{proposition}

\end{document}